\documentclass[a4paper,10pt]{article}
\usepackage{amssymb,amsthm}
\usepackage{color}
\usepackage{graphicx}
\theoremstyle{definition}
\newtheorem{definition}{Definition}
\newtheorem{example}{Example}

\theoremstyle{plain}
\newtheorem{theorem}{Theorem}
\newtheorem{proposition}{Proposition}

\newcommand{\ZZ}{\mathbb{Z}}
\newcommand{\NN}{\mathbb{N}}

\title{New Semifield Planes of order 81\footnote{This is a revised paper (original title: \emph{New semifield planes of orders 64 and 81)}. During the referee process the authors were informed that U. Dempwolff had previously and independently obtained a classification of finite semifields of order 81 (\emph{Semifield Planes of Order 81},
    {J. of Geometry} (to appear)). So, the current version (unpublished) contains only the results concerning 81 element finite semifields. For the results on order 64 the reader is referred to \emph{Classification on 64-element finite semifields} (I.F. R\'ua, E.F. Combaro and J. Ranilla) 	arXiv:0807.2994v1}}
\author{ I.F. R\'
ua\thanks{Departamento de Matem\'aticas, Universidad de Oviedo, rua@uniovi.es .
Partially supported by FICYT (IB05-186) and MEC - MTM - 2007 - 67884 Ð C04 - 01}\and El\'\i as F. Combarro\thanks{Departmento de Inform\'atica,
Universidad de Oviedo, elias@aic.uniovi.es . Partially supported by MEC -TIN - 2007 - 61273}}
\date{}

\begin{document}

\maketitle

\begin{abstract}
A finite semifield $D$ is a finite nonassociative ring with identity such that the set $D^*=D\setminus\{0\}$ is closed under the product. From any finite semifield a projective plane can be constructed. In this paper we obtain new semifield planes of order 81 by means of computational methods. These computer-assisted results yield to a complete classification (up to isotopy) of 81-element finite semifields.
   
\end{abstract}

\section{Introduction}

A \textbf{finite semifield} (or finite division ring) $D$ is a
finite nonassociative ring with identity such that the set
$D^*=D\setminus\{0\}$ is closed under the product, i.e., it is a loop \cite{Knu65,Cor99}.
Finite semifields have been traditionally considered in the context of finite
geometries since they coordinatize projective semifield planes \cite{Hall}. Recent
applications to coding theory \cite{Calderbank3,codigosdesemicuerpos,Symplectic}, combinatorics and graph theory \cite{ISSAC}, have
broaden the potential interest in these rings.

Because of their diversity, the obtention of general theoretical
algebraic results seems to be a rather difficult (and challenging)
task. On the other hand, because of their finiteness, computational
methods can be naturally considered in the study of these objects. 
So, the classification of finite semifields of a given order is a rather natural problem to use computations. For instance, computers were used in the classification up to isotopy of finite semifields or order 32 \cite{Walker,Knu65}. This computer-assisted classification is equivalent to the classification of the corresponding projective semifield planes up to isomorphism \cite{Alb60}.

In this paper we obtain a classification up to isotopy of finite semifields with \textbf{81 elements}. It turns out that approximately one half of the 81-element semifield planes were previously unknown.

The structure of the paper is as follows. In $\S 2$, basic
properties of finite semifields are reviewed. Finally, in $\S 3$, a complete description of 81-element finite semifields is given.

\section{Preliminaries}

In this section we collect definitions and facts on finite
semifields. Proofs of these results can be found, for instance, in
\cite{Knu65,Cor99}.

\begin{definition}
    A finite nonassociative ring $D$ is called \textbf{presemifield},
    if the set of nonzero elements $D^*$ is closed under the
    product.
    If $D$ has an identity element, then it is called \textbf{finite semifield}.
\end{definition}

If $D$ is a finite semifield, then $D^*$ is a multiplicative loop.
That is, there exists an element $e\in D^*$ (the identity of $D$)
such that $ex=xe=x$, for all $x\in D$ and, for all $a,b\in D^*$,
  the equation $ax=b$
(resp. $xa=b$) has a unique solution.

Besides finite fields (which are obviously finite semifields), finite semifields were first considered by L.E. Dickson
\cite{Dickson1} and were deeply studied by A.A. Albert
\cite{Alb52,Alb58,Alb60,Alb61}. The term ``finite semifield'' was
introduced in 1965 by D.E. Knuth \cite{Knu65}. These rings play an
important role in the study of certain projective planes, called \emph{semifield planes}
\cite{Hall,Knu65}. Recently, applications of finite semifields to
coding theory have been also considered
\cite{Calderbank3,codigosdesemicuerpos,Symplectic}. Also, connections to combinatorics and graph theory have been found \cite{ISSAC}.

\begin{proposition}
    The characteristic of a finite presemifield $D$ is a prime number $p$, and $D$ is a finite-dimensional algebra over $GF(q)$ ($q=p^c$) of dimension $d$, for some $c,d\in \NN$, so that 
    $|D|=q^d$.
    If $D$ is a finite semifield, then $GF(q)$ can be chosen to be its associative-commutative center $Z(D)$.
\end{proposition}

\begin{definition}
    A finite semifield $D$ is called \textbf{proper}, if it is not
    associative.
\end{definition}

Because of Wedderburn's theorem \cite{Wed}, finite semifields are either proper or finite fields. The existence of proper semifields is guaranteed by the following result.

\begin{theorem}
    If $D$ is a proper semifield of dimension $d$ over $Z(D)$, then
    $d\ge 3$ and $|D|\ge 16$.
    Moreover, for any prime power $p^n\ge 16$, with $n\ge 3$, there exists a proper semifield $D$ of
    cardinality $p^n$.
\end{theorem}

The definition of isomorphism of  presemifields is the usual one for algebras, and classification of finite semifields up to isomorphism can be naturally considered. Because of the conections to finite geometries, we must also consider the following notion.

\begin{definition}
	If $D_1,D_2$ are two presemifields, then an \textbf{isotopy} between $D_1$ and $D_2$ is a triple $(F,G,H)$ of bijective linear maps $D_1\to D_2$ such that
	$$H(ab)=F(a)G(b)\; \forall a,b\in D_1.$$
\end{definition}

It is clear that any isomorphism between two presemifields is an isotopy, but the converse is not necessarily true. It can be shown that any presemifield is isotopic to a finite semifield \cite[Theorem 4.5.4]{Knu65}.
From any presemifield $D$ a projective plane $\mathcal P(D)$ can be constructed (see \cite{Hall,Knu65} for the details of this construction). Theorem 6 in \cite{Alb60} shows that isotopy of finite semifields is the algebraic translation of the isomorphism between the corresponding projective planes. So, two finite semifields $D_1,D_2$ are isotopic if, and only if, the projective planes $\mathcal P(D_1),\mathcal P(D_2)$ are isomorphic. The set of isotopies between a finite semifield and itself is a group under composition, called the \textbf{autotopy} group.

Given a finite semifield $D$, it is possible to construct the set $\mathcal D$ of all its isotopic but non necessarily isomorphic finite semifields. It can be obtained from the \textbf{principal isotopes} of $D$ \cite{Knu65}. A principal isotope of $D$ is a finite semifield $D_{(y,z)}$ (where $y,z\in D^*$) such that $(D_{(y,z)},+)=(D,+)$ and multiplication is given by the rule
$$a\cdot b=R_z^{-1}(a)L_y^{-1}(b)\; \forall a,b\in D$$
where $R_z,L_y:D\to D$ are the maps $R_z(a)=az,L_y(a)=ya$, for all $a\in D$.
Moreover, there is a relation between the order of $\hbox{At}(D)$, the autotopy group of $D$, and the orders of the automorphism groups of the elements in $\mathcal D$ \cite[Theorem 3.3.4]{Knu65}.

\begin{theorem}
	If $D$ is a finite semifield, and $\mathcal D$ is the set of all nonisomorphic semifields isotopic to $D$, then
	$$(|D|-1)^2=|\emph{At}(D)|\sum_{E\in \mathcal D}\frac{1}{|\emph{Aut}(E)|}$$
	The sum of the right term will be called \emph{the Semifield/Automorphism (S/A) sum}.
\end{theorem}

If $\mathcal B=[x_1,\dots,x_d]$ is a $GF(q)$-basis of a presemifield $D$, then there exists a unique set of constants $\mathbf A_{D,\mathcal B}=\{A_{i_1i_2i_3}\}_{i_1,i_2,i_3=1}^d\subseteq GF(q)$ such that 
$$x_{i_1}x_{i_2}=\sum_{i_3=1}^d{A_{i_1i_2i_3}}x_{i_3}\; \forall i_1,i_2\in\{1,\dots,d\}$$
It is called \textbf{3-cube} corresponding to $D$ with respect to the basis $\mathcal B$ \cite{Knu65}. This set is also known as multiplication table, and it completely determines the multiplication in $D$.
If $\mathcal B'$ is a second basis of $D$ (where $\mathcal B'=\mathcal BP^t$, with $P\in \hbox{GL}(d,q)$), then the relation between the 3-cubes $\mathbf A_{D,\mathcal B},\mathbf A_{D,\mathcal B'}$ is
$$\mathbf A_{D,\mathcal B'}=[P,P,P^{-t}]\times\mathbf A_{D,\mathcal B}$$
where $-t$ denotes inverse transpose, and
$$([F^1,F^2,F^3]\times \mathbf A_{D,\mathcal B})_{i_1i_2i_3}=\sum_{x_1,x_2,x_3=1}^d 
F^1_{i_1x_1}F^2_{i_2x_2}{F^3}_{i_3x_3}A_{x_1x_2x_3}$$
for all $F{^1},F{^2},F{^3}\in \hbox{GL}(d,q)$ \cite[(4.11)]{Knu65}.

A remarkable fact is that permutation of the indexes of a 3-cube preserves the absence of nonzero divisors. Namely, if $D$ is a presemifield, and $\sigma\in S_3$ (the symmetric group on the set $\{1,2,3\}$), then the set
$$\mathbf A_{D,\mathcal B}^\sigma=\{A_{i_{\sigma(1)}i_{\sigma(2)}i_{\sigma(3)}}\}_{i_1,i_2,i_3=1}^d\subseteq GF(q)$$ is the 3-cube of a $GF(q)$-algebra $D_{\mathcal B}^\sigma$ which has not zero divisors \cite[Theorem 4.3.1]{Knu65}. 
Notice that, in general, different bases $\mathcal B,\mathcal B'$ lead to nonisomorphic presemifields $D_{\mathcal B}^\sigma,D_{\mathcal B'}^\sigma$. However, these presemifields are always isotopic. This is a consequence of the following two propositions.

\begin{proposition}\cite[Theorem 4.4.2]{Knu65}
Let $D_1,D_2$ be two presemifields of dimension $d$ over a finite field $GF(q)$, and let $\mathcal B_1,\mathcal B_2$  be two $GF(q)$-bases of $D_1$ and $D_2$. Then, $D_1$ and $D_2$ are isotopic if and only if there exist  three nonsingular maps
$F{^1},F{^2},F{^3}\in \hbox{GL}(d,q)$ such that 
$\mathbf A_{D_2,\mathcal B_2}=[F{^1},F{^2},F{^3}]\times\mathbf A_{D_1,\mathcal B_1}$ .
\end{proposition}

\begin{proposition}\cite[Theorem 4.2.3]{Knu65}\label{lema}
	Let $D$ be a presemifield of dimension $d$ over $GF(q)$, and let $\mathcal B$ be a $GF(q)$-basis of $D$. Then, for all $\sigma \in S_3$, and for all $F{^1},F{^2},F{^3}\in \hbox{GL}(d,q)$,
		 $$([F{^1},F{^2},F{^3}]\times \mathbf A_{D,\mathcal B})^\sigma=[F{^{\sigma^{-1}(1)}},F{^{\sigma^{-1}(2)}},F{^{\sigma^{-1}(3)}}]\times \mathbf A^\sigma_{D,\mathcal B}$$
\end{proposition}

The number of projective planes that can be constructed from a given finite semifield $D$ using the transformation of the group $S_3$ is at most six \cite[Theorem 5.2.1]{Knu65}. Actually, $S_3$ acts on the set of semifield planes of a given order. So, the classification of finite semifields can be reduced to the classification of the corresponding projective planes up to isotopy and, ultimately, this can be reduced to the classification of semifield planes up to the action of the group $S_3$. In this setting, we will consider a plane as \emph{new} when no \emph{known}\footnote{We have
considered as \emph{known} semifields those appearing in the, up to our knowledge, last survey on the topic, \cite{Kantor}} finite semifield coordinatizes a plane in its $S_3$-orbit.

Invariance under the transformation by $S_3$ leads to the following definition.

\begin{definition}
	Let $D$ be a finite presemifield $d$-dimensional over $GF(q)$, and let $\sigma\in S_3$. Then, $D$ is called:
	\begin{itemize}
		\item {\bf $\sigma$-invariant}, if there exists a $GF(q)$-basis $\mathcal B$ of $D$, such that $\mathbf A^\sigma_{D,\mathcal B}=\mathbf A_{D,\mathcal B}$.
		\item {\bf $\sigma$-isotopic}, if there exists a $GF(q)$-basis $\mathcal B$ of $D$, such that $D^\sigma_{\mathcal B}$ is isotopic to $D$.
	\end{itemize}
\end{definition}

Clearly, if a presemifield is $\sigma$-invariant, then it is $\sigma$-isotopic.
	Different well-known notions can be rephrased in terms of this definition.

\begin{example}
 A finite presemifield:
	\begin{itemize}
		\item Is commutative iff it is $(1,2)$-invariant.
		\item Is symplectic \cite{codigosdesemicuerpos} iff it is $(1,3)$-invariant.
		\item Induces a self-dual plane \cite{Knu65} iff it is $(1,2)$-isotopic.
		\item Induces a self-transpose plane \cite{Kan03} iff it is $(1,3)$-isotopic.
	\end{itemize}
\end{example}

It is clear that, if $D$ is $\sigma$-invariant ($\mathbf A^\sigma_{D,\mathcal B}=\mathbf A_{D,\mathcal B}$, where $\mathcal B$ is a $GF(q)$-basis of $D$), then 
 $\mathbf A^{\sigma\tau}_{D,\mathcal B}=\mathbf A^\tau_{D,\mathcal B}$, for all $\tau \in S_3$. Similarly, 
 as a consequence of Proposition \ref{lema}, if $D$ is $\sigma$-isotopic (i.e., $D^\sigma_{\mathcal B}$ is an isotope of $D_{\mathcal B}$, where $\mathcal B$ is a $GF(q)$-basis of $D$), then $D^{\sigma\tau}_{\mathcal B}$ is an isotope of $D^\tau_{\mathcal B}$.
 This fact will be useful in the description of all possible situations depending on the orbits of $S_3$ (cf. \cite[Proposition 3.8]{Kan03},\cite[Corollary 5.2.2]{Knu65}).
 We shall use a graphical representation to distinguish between the different cases. 
The vertices of an hexagon will depict the six different planes obtained from a given finite semifield (cf. \cite[Theorem 5.2.1]{Knu65}).

$$\hbox{ }\hspace{.5cm}\includegraphics[width=3cm,height=3cm]{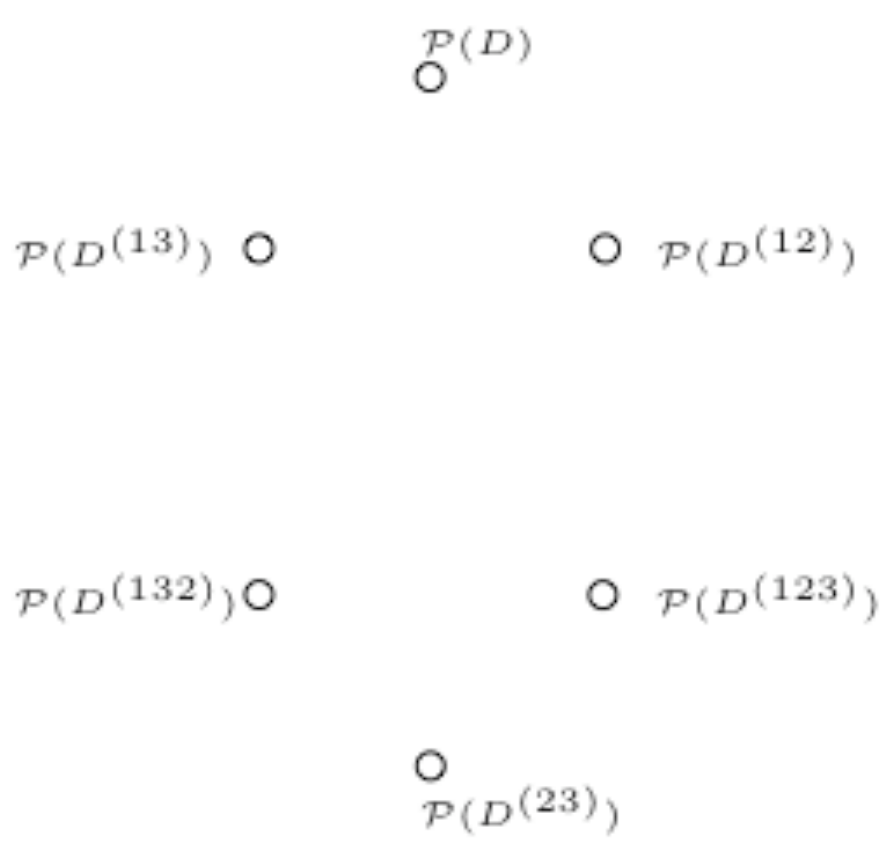}$$

A dotted line between two planes shows that the corresponding finite semifields are $\sigma$-isotopic (where $\sigma$ is a suitable element of $S_3$). A continous line stands for $\sigma$-invariance of one of the corresponding coordinatizing finite presemifields.
Here are the different posibilities (cf. \cite[Corollary 5.2.2]{Knu65},\cite[Proposition 3.8]{Kan03}).

\begin{itemize}
	\item A unique orbit consisting of six planes.
	
		$$\includegraphics[width=1cm,height=1cm]{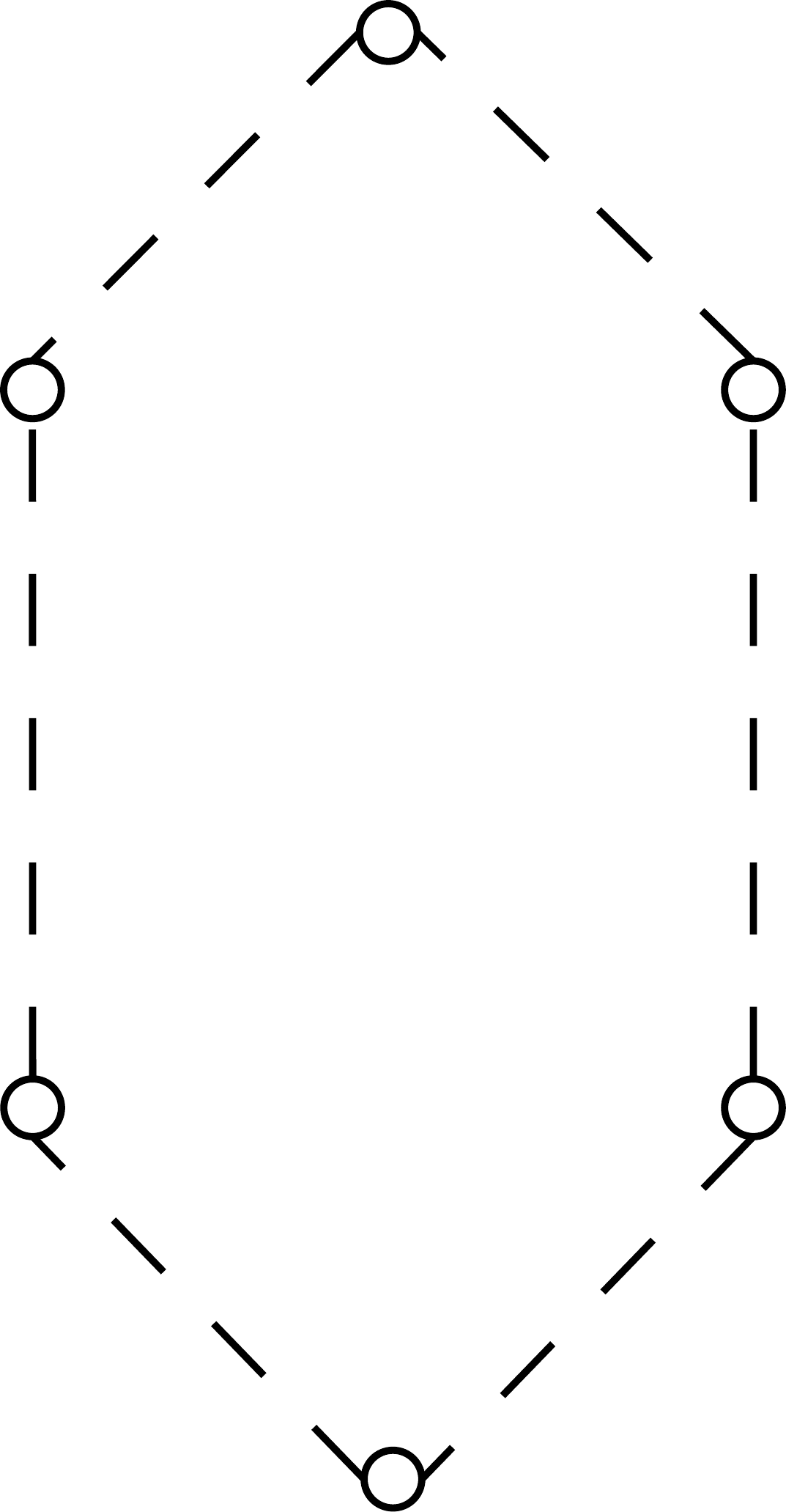}\hspace{2cm}\includegraphics[width=1cm,height=1cm]{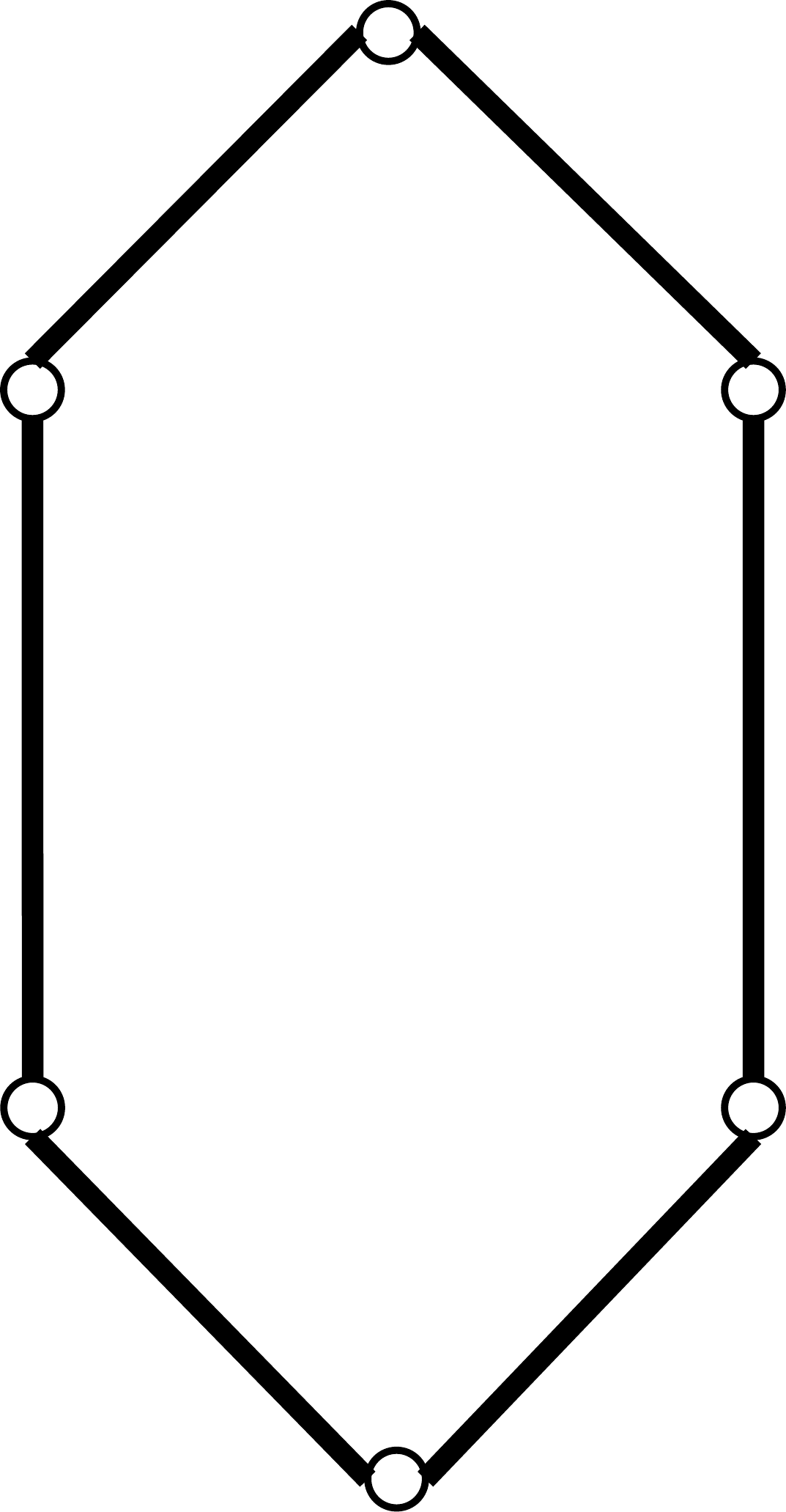}$$
	
	\item Two orbits of three planes each.
	
		$$\includegraphics[width=1cm,height=1cm]{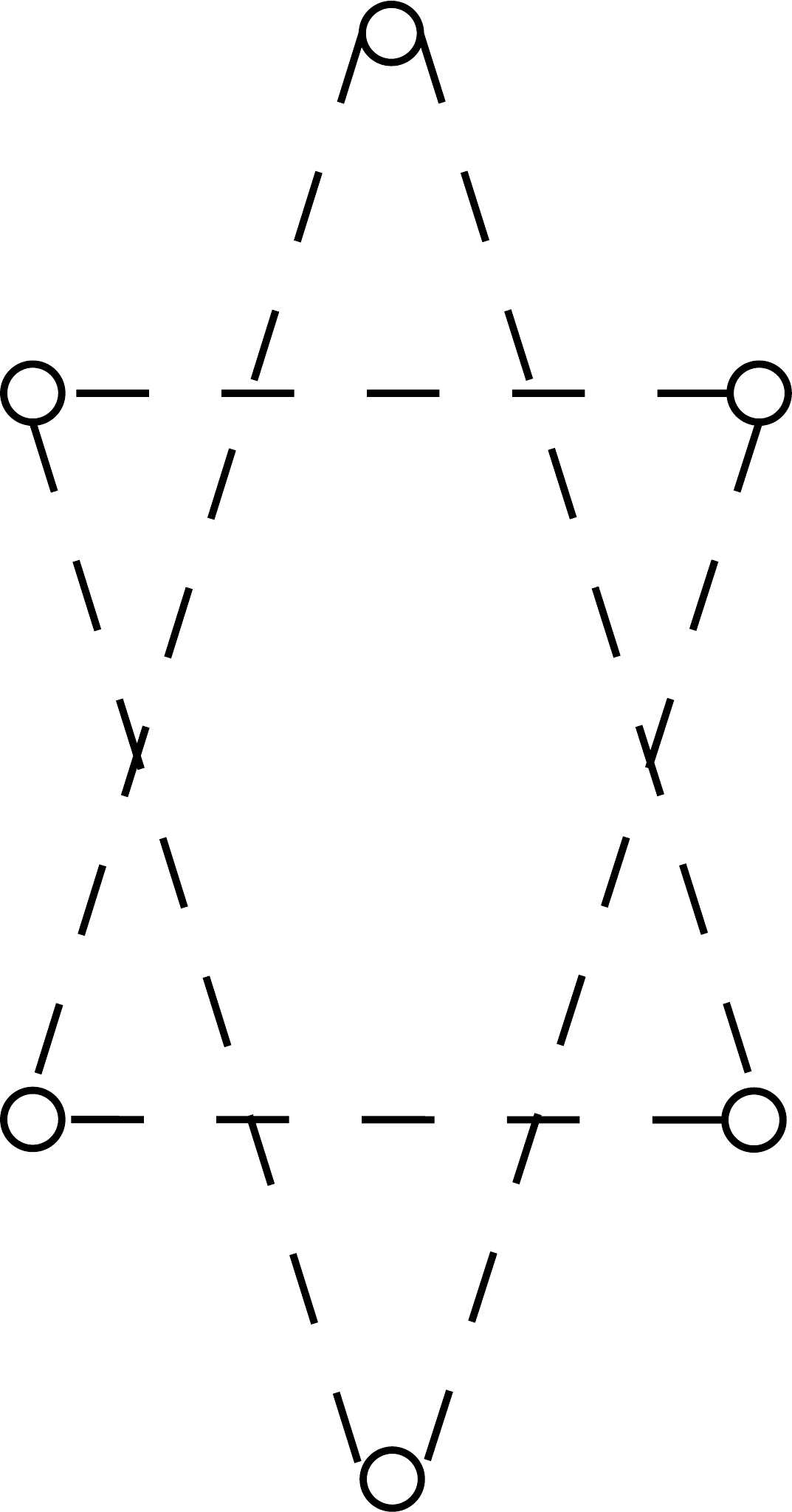}\hspace{2cm}\includegraphics[width=1cm,height=1cm]{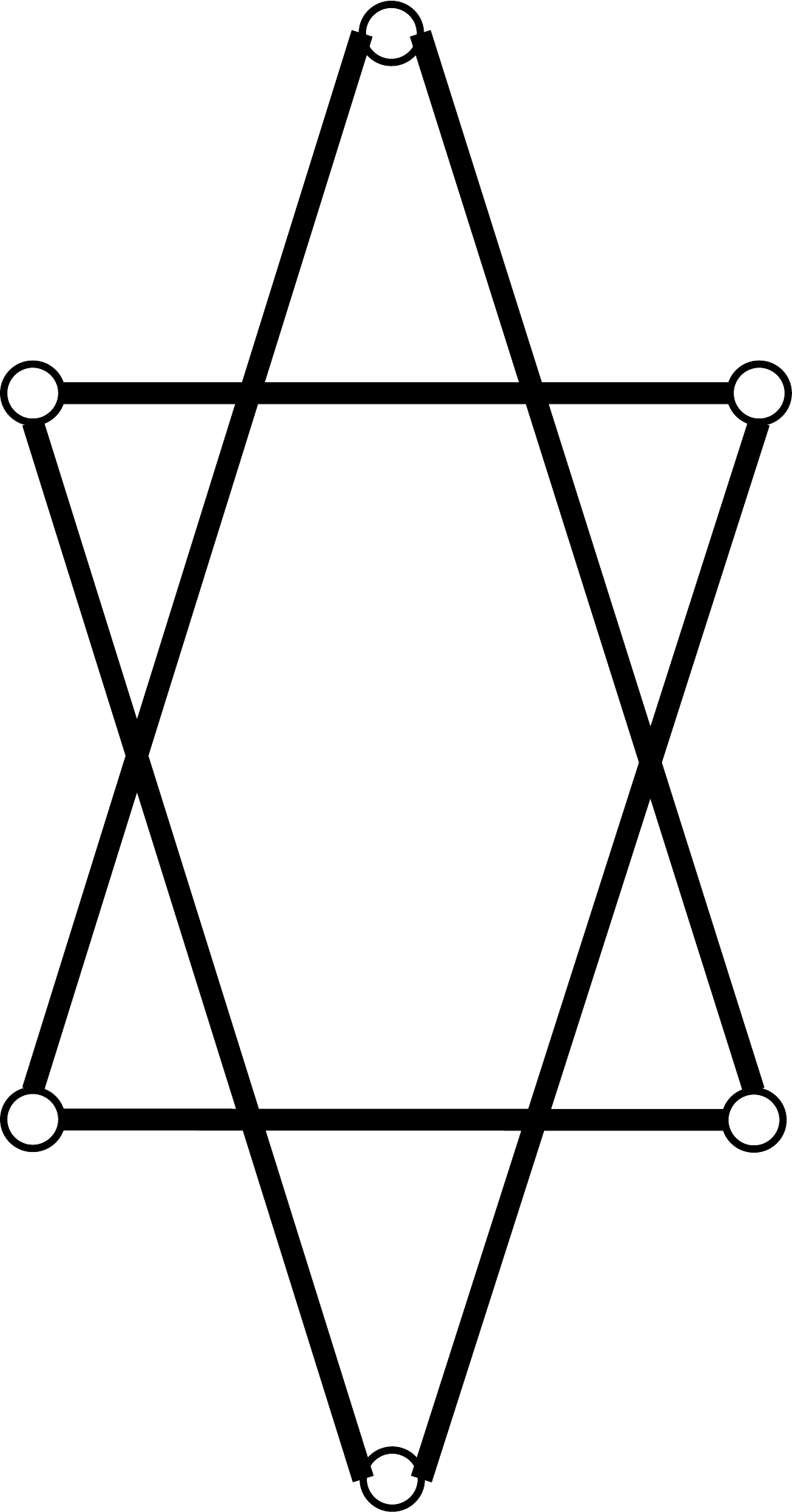}$$

	\item Three orbits of two planes each.
	
		$$\includegraphics[width=1cm,height=1cm]{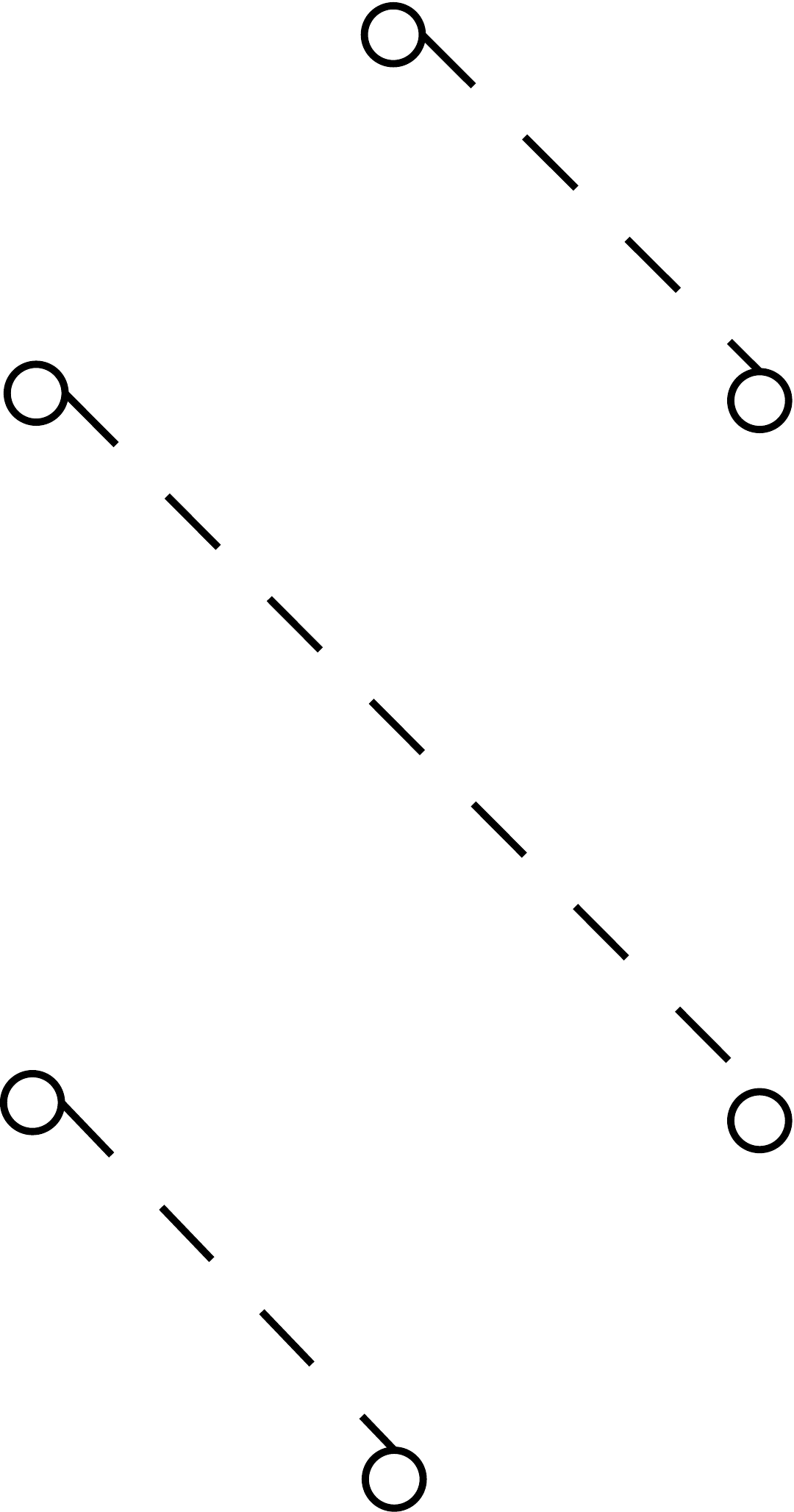}\hspace{2cm}
		\includegraphics[width=1cm,height=1cm]{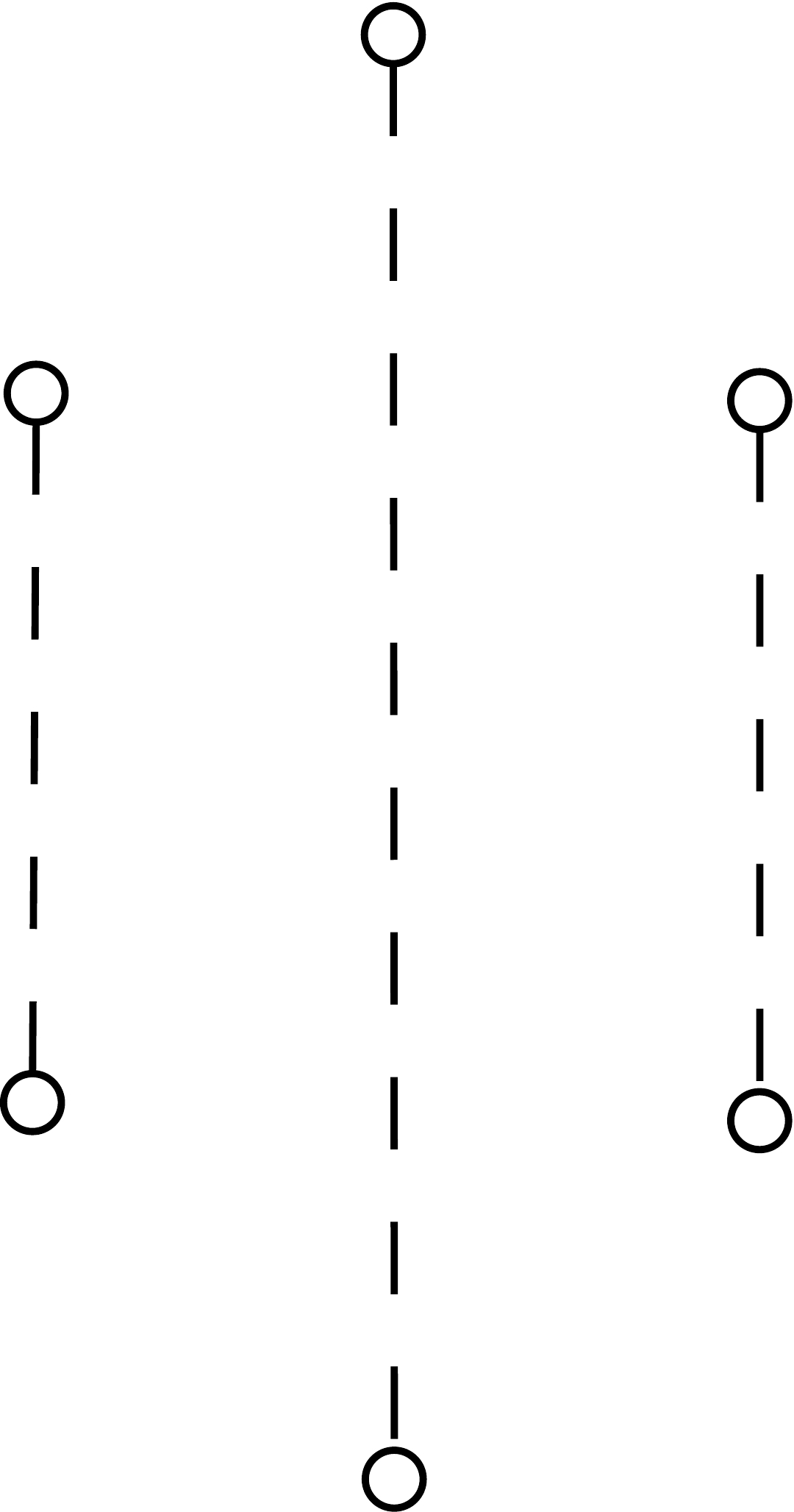}\hspace{2cm}
		\includegraphics[width=1cm,height=1cm]{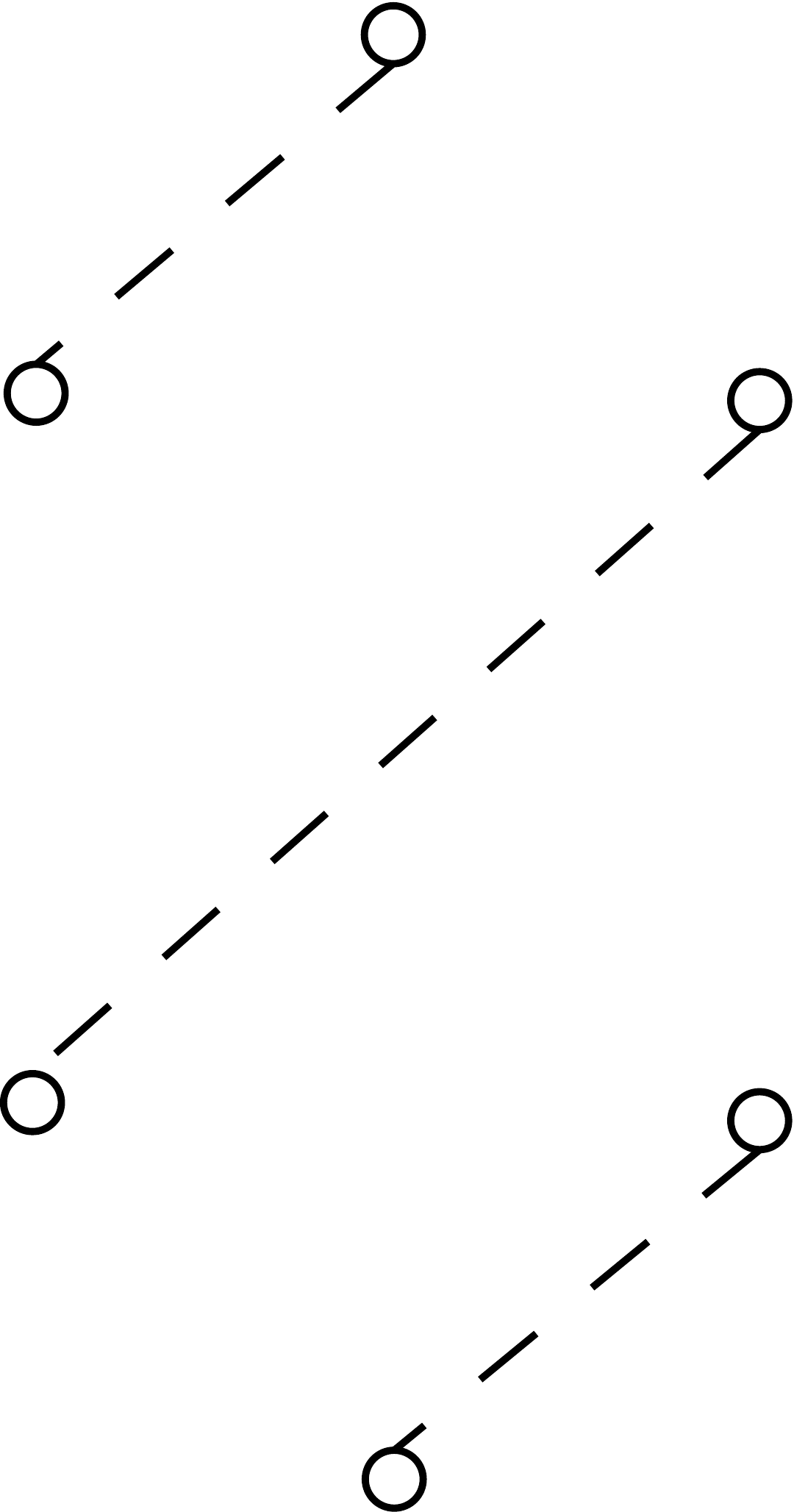}$$
		$$
		\includegraphics[width=1cm,height=1cm]{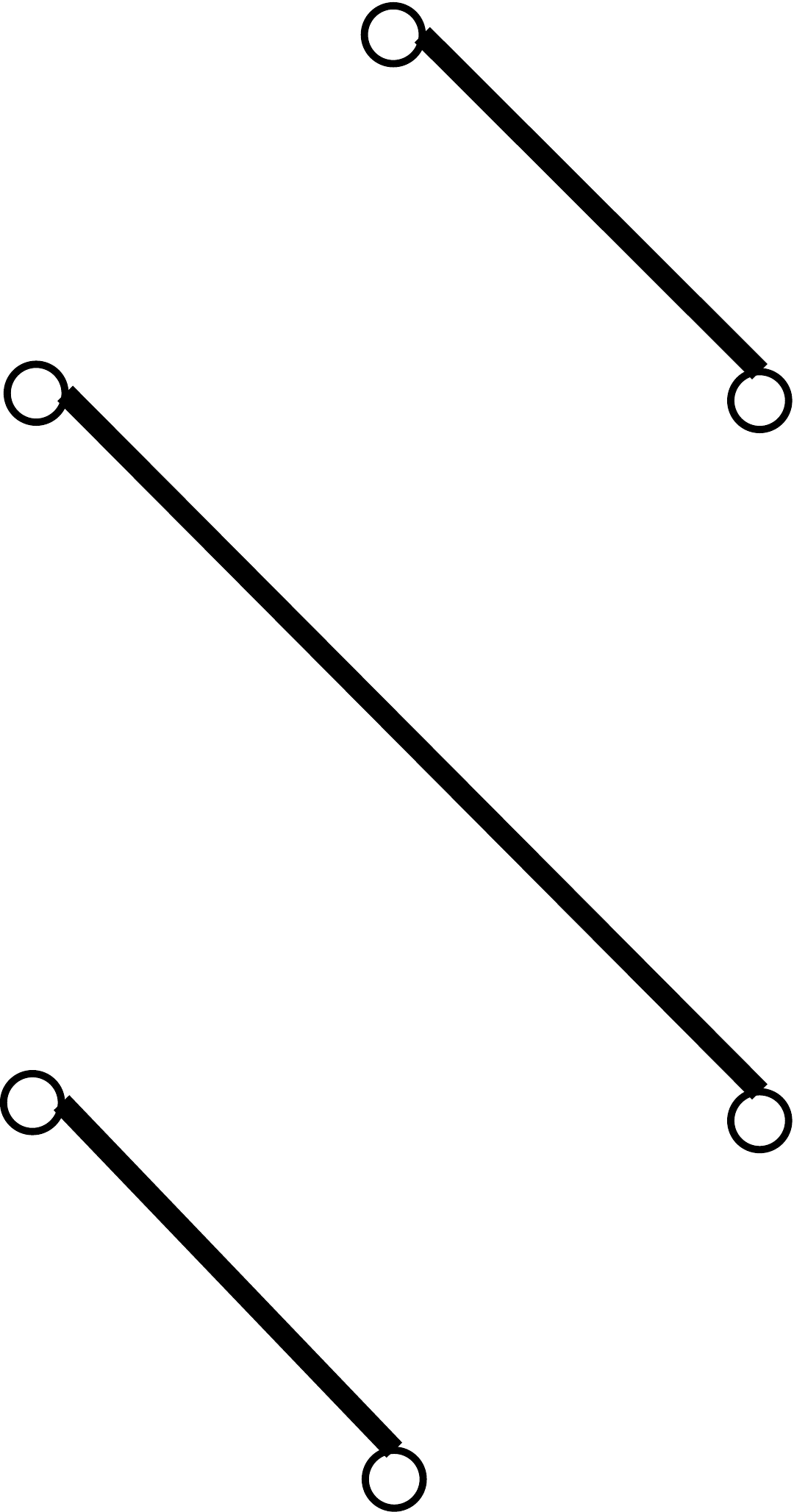}\hspace{2cm}\includegraphics[width=1cm,height=1cm]{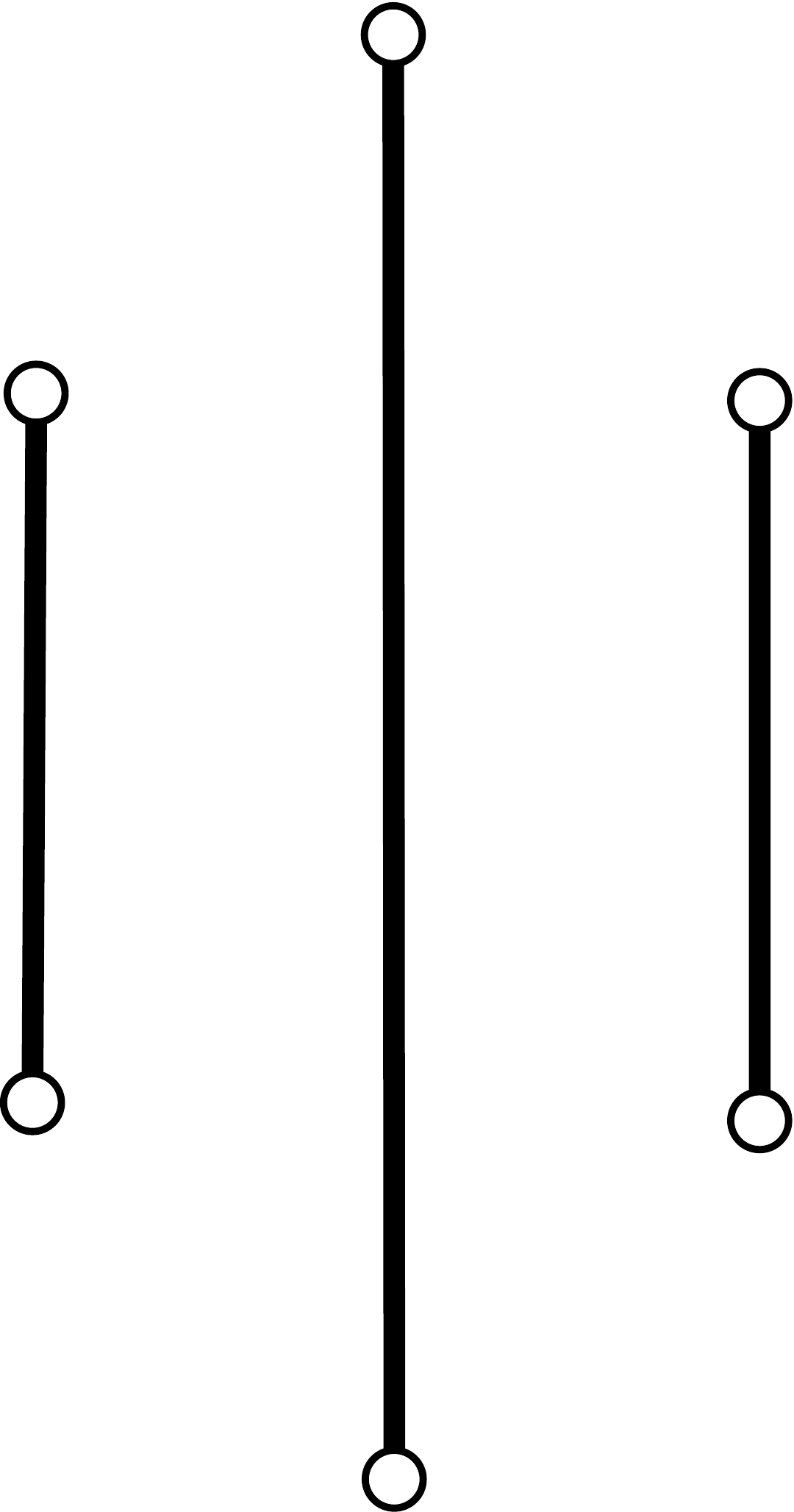}
	\hspace{2cm}\includegraphics[width=1cm,height=1cm]{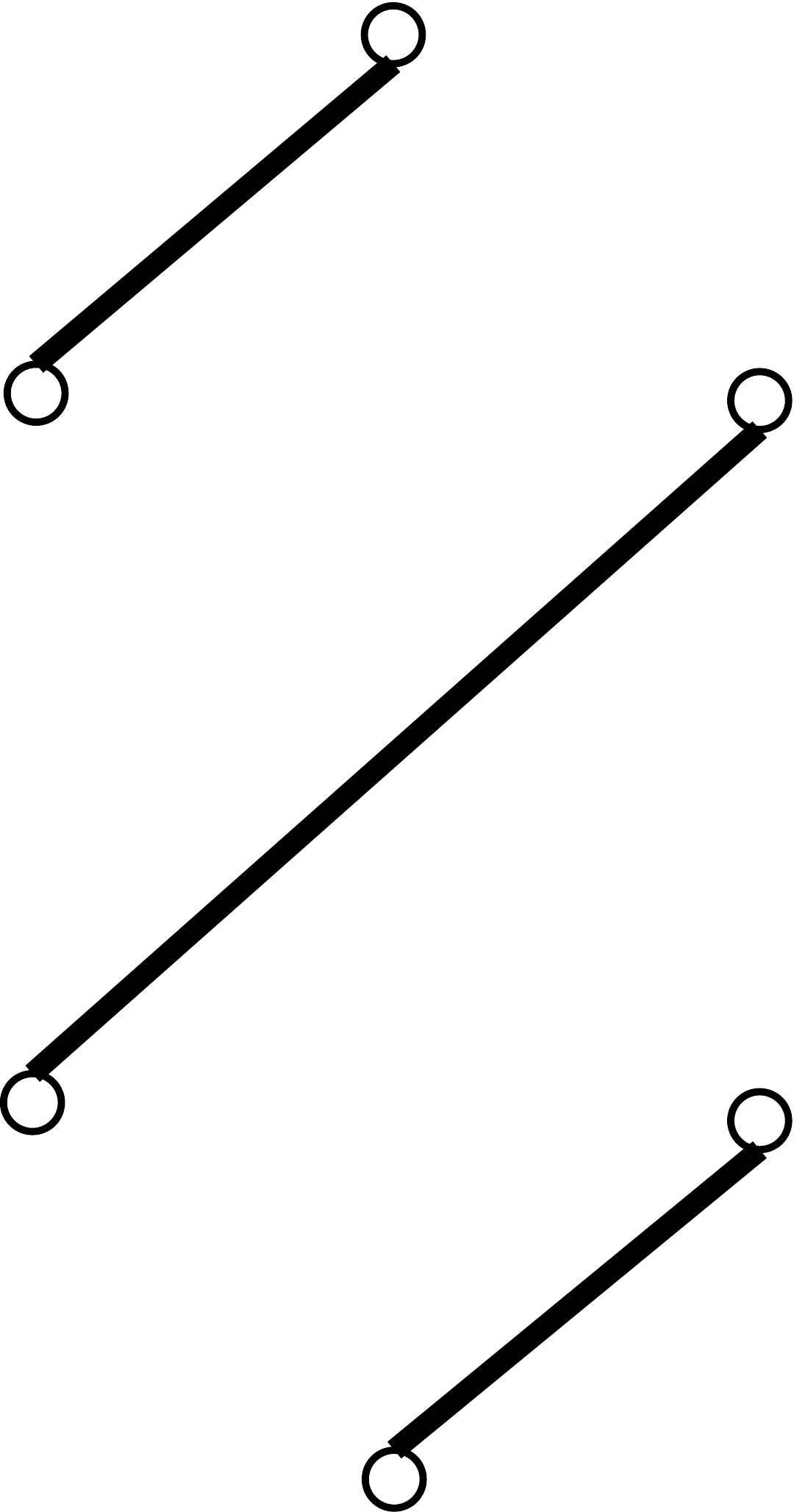}$$
	
	\item Six different planes.
	
	$$\includegraphics[width=1cm,height=1cm]{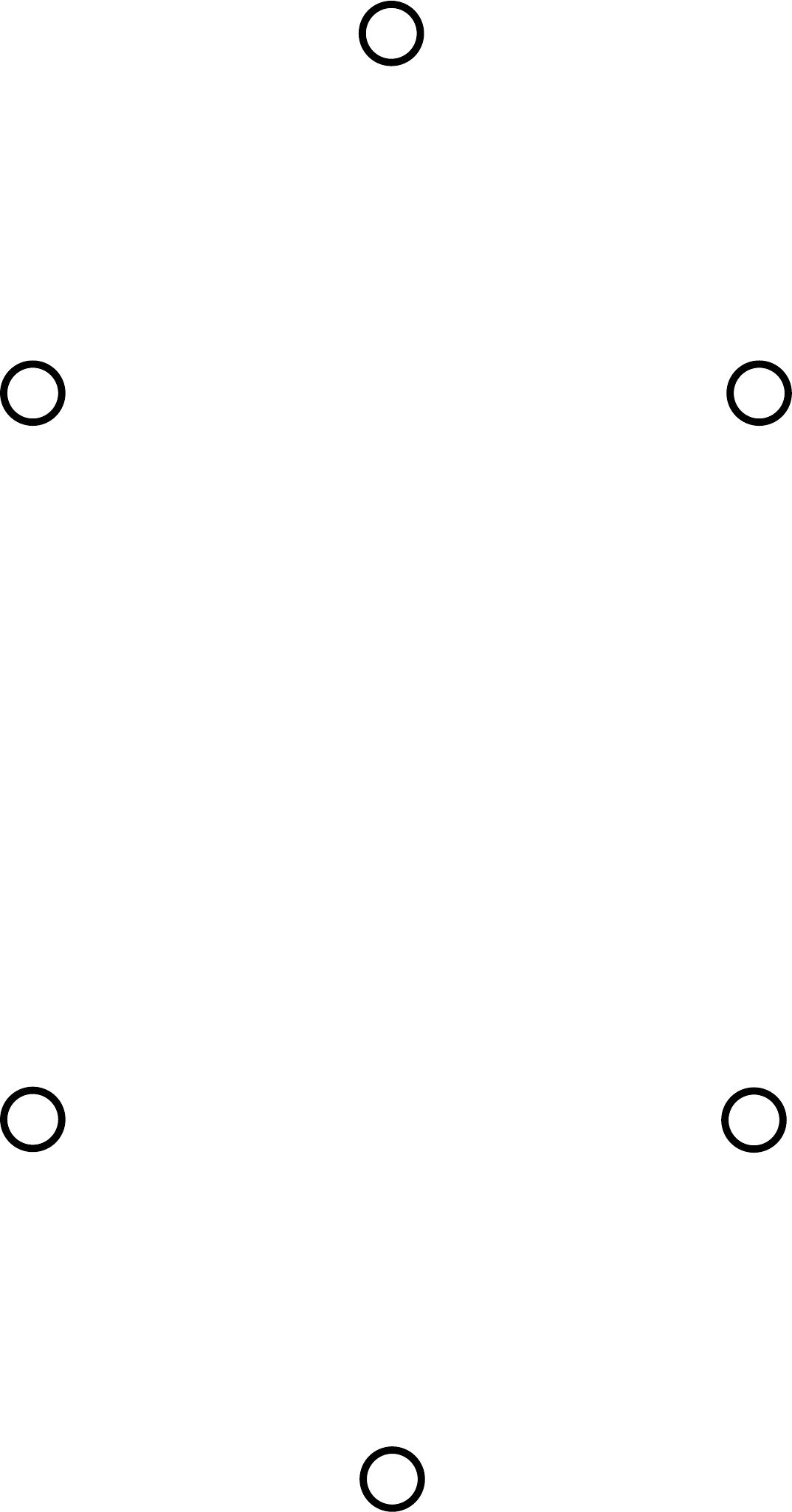}$$
		
\end{itemize}

We finish this section recalling that 
the construction of finite semifields of a given order can be rephrased
as a matrix problem \cite[Proposition 3]{Hentzel}.

\begin{proposition}\label{matrices}
    There exists a finite semifield $D$ of
    dimension $d$ over its center $Z(D)\supseteq GF(q)$ if, and only if, there exists a
    set of $d$ matrices $\{A_1,\dots,A_d\}\subseteq GL(d,q)$ such
    that:
    \begin{enumerate}
        \item $A_1$ is the identity matrix;
        \item $\sum_{i=1}^d\lambda_i A_i\in GL(d,q)$, for all
        \textbf{non-zero tuples}
        $(\lambda_1,\dots,\lambda_d)\in \prod^d GF(q)$, that is,
        $(\lambda_1,\dots,\lambda_d)\not=\{\overrightarrow 0\}$.
        \item The first column of the matrix $A_i$ is the column vector
        $e_i^\downarrow$ with a $1$ in
    the $i$-th position, and $0$ everywhere else.
    \end{enumerate}
\end{proposition}

\section{New Semifield Planes of order 81: a classification}

This section is devoted to the study of semifield planes of order 81. All these planes have been obtained by exhaustive search of coordinatizing semifields, and a complete classification has been later achieved.

Let us begin by fixing some notation. If $D$ is a finite semifield
with 81 elements, then it is a finite nonassociative algebra of
dimension 4 over $\ZZ_3$. From \cite{Hentzel}[Proposition 3] there exists a set of matrices
$\{A_1=I_4,A_2,A_3,A_4\}\subseteq GL(4,3)$ with the first
column of $A_i$ equal to the vector $e_i^\downarrow$
($i=1,\dots,4$), such that $\sum_{i=1}^4\lambda_i A_i\in GL(4,3)$,
for all non-zero tuples
        $(\lambda_1,\lambda_2,\lambda_3,\lambda_4)\in \prod^4 \ZZ_3$. 
        These matrices are the
        coordinate matrices of the maps $R_{x_i}$ where
        $\mathcal B=\{x_1=e,x_2,x_3,x_4\}$ is a $\ZZ_3$-basis of $D$.
Hence, the existence of 81-element semifields is reduced to the existence of sets of 4 matrices satisfying certain conditions.
The first of these matrices is always the identity matrix. Let us now show that the second matrix can also be chosen from a small amount of matrices.

\begin{proposition}\label{base81}
If $D$ is a finite semifield with 81 elements,
then there is a $\ZZ_3$-basis $\mathcal
B=\{x_1=e,x_2,x_3,x_4\}$ of $D$ such that the coordinate
matrix of $R_{x_2}$ is a companion matrix of one of the following polynomials:
$$(1)\ x^4+x+2\;  (2)\ x^4+2x+2\; (3)\ x^4+x^3+2$$
$$ (4)\ x^4+x^3+x^2+2x+2\; (5)\ x^4+x^3+2x^2+2x+2\;  (6)\ x^4+2x^3+2$$$$ (7)\ x^4+2x^3+x^2+x+2\; (8)\ x^4+2x^3+2x^2+x+2$$
\end{proposition}

\begin{proof}
	From \cite{Hentzel}[Theorem 2] the finite semifield $D$ is left and right primitive \cite{Wene91}. Therefore, there exists an element $x_2\in D$ such that the characteristic polynomial $p_{x_2}(x)$ of the linear transformation $R_{x_2}$ is a primitive polynomial of degree 4 over $\ZZ_3$ \cite{Hentzel}[Proposition 2]. So, in view 
of the list of the primitive polynomials of order 4 over $\ZZ_3$ \cite{Lidl} it must be one of the eight listed above. Moreover, in particular, $p_{x_2}(z)$ is irreducible in $\ZZ_3[z]$,  and so $\mathcal B=\{1,x_2,{x_2}^{2)},{x_2}^{3)}\}$ is a linearly independent set, i.e., it is a $\ZZ_3$-basis of $D$. The coordinate matrix of $R_{x_2}$ with respect to such a basis is clearly a companion matrix. 
\end{proof}

We have used this property to design a search algorithm for 81-element finite semifields. 
Our algorithm, which is a rather standard backtracking method (cf.
\cite{Walker}), is written below with the help of two auxiliary functions. The
first one ({\it Complete}) enumerates all valid semifields with given
initial matrices. The second one ({\it Complete2}) enumerates all
valid semifield with given initial matrices and some columns of the
next matrix.
The function {\it Complete} uses the already known matrices to create the initial columns of the next
matrix and then calls {\it Complete2}. This second function, in turn,
recursively adds columns to the incomplete matrix (backtracking if
necessary) and then calls {\it Complete} with new matrix.

\vspace{0.3cm}\hrule\vspace{0.3cm}
\centerline{{\bf Algorithm 1:} Search algorithm for finite semifields}
\vspace{0.3cm}\hrule\vspace{0.3cm}

\textbf{$\bullet $ Input:} Characteristic $p$ and dimension $n$ of the semifields, second matrix $A_2$

\textbf{$\bullet$ Output:} List of matrices representing all the semifields with second matrix $A_2$, of the given characteristic and
  dimension

\textbf{$\bullet$ Procedure:}

{$\hbox{ }$ Create} an empty list of matrices $L$

{$\hbox{ }$ Insert} the identity $I$ in $L$

{$\hbox{ }$ Insert} $A_2$ in $L$

{$\hbox{ }$ Call} $Complete(L,p,n)$
\vspace{0.3cm}\hrule\vspace{0.3cm}

\vspace{0.3cm}\hrule\vspace{0.3cm}
\centerline{{\bf Algorithm 2:} Function $Complete$ }
\vspace{0.3cm}\hrule\vspace{0.3cm}

\textbf{$\bullet $ Input:} A list of matrices $L$, the characteristic $p$ and the dimension $n$

\textbf{$\bullet$ Output:} List of matrices representing all the semifields with initial matrices $L$, of the given characteristic and
  dimension

\textbf{$\bullet$ Procedure:}

$\hbox{ }$ $m \leftarrow $ size of $L$

\textbf{$\hbox{ }$ if} $m$ is equal to $n$ \textbf{then}

 \textbf{\hspace{.7cm} return} $L$
 
 \textbf{$\hbox{ }$ end}
 
  \textbf{$\hbox{ }$ else}
  
 \textbf{\hspace{.7cm}}	Create a matrix $M$ of 1 column 
 
  \textbf{\hspace{.7cm}}
	Set the first column of $M$ equal to the $(m+1)$-th column of the identity
	
	 \textbf{\hspace{.7cm}}
	Call $Complete2(L,M,p,n)$
	
	 \textbf{$\hbox{ }$ end}
\vspace{0.3cm}\hrule\vspace{0.3cm}

\vspace{0.3cm}\hrule\vspace{0.3cm}
\centerline{{\bf Algorithm 3:} Function $Complete2$}
\vspace{0.3cm}\hrule\vspace{0.3cm}

\textbf{$\bullet $ Input:} A list of matrices $L$, a truncated matrix $M$, the characteristic $p$ and the dimension $n$

\textbf{$\bullet$ Output:} List of matrices representing all the semifields with initial matrices $L$ and $M$, of the given characteristic and
  dimension

\textbf{$\bullet$ Procedure:}

$\hbox{ }$ $k \leftarrow $ number of columns of $M$

\textbf{$\hbox{ }$ if} $k$ is equal to $n$ \textbf{then}

    \textbf{\hspace{.7cm}} Insert $M$ in $L$ 
    
  \textbf{\hspace{.7cm}} Call $Complete(L,p,n)$ 

\textbf{$\hbox{ }$ end}

\textbf{$\hbox{ }$ else}

{\hspace{.7cm}} Compute $C$, the list of columns $c$ such that the join of $M$ and $c$ is 

{\hspace{.7cm}}  linearly independent 
	of the matrices of $L$ (truncated at the $k+1$ first 
	
	{\hspace{.7cm}} 
	columns)

 \textbf{\hspace{.7cm} for} each $c$ in $C$ \textbf{do}

        \textbf{\hspace{1.3cm}} Join $c$ to $M$, as its $k+1$-th column
                
      \textbf{\hspace{1.3cm}}  Call $Complete2(L,M,p,n)$
      
      \textbf{\hspace{1.3cm}} Remove $c$ from $M$
      
       \textbf{\hspace{.7cm} end}
     
     \textbf{$\hbox{ }$ end}
\vspace{0.3cm}\hrule\vspace{0.3cm}

We have implemented
this algorithm in the language C++. The run time on a 
2GHz desktop computer was about 150 seconds (for each separate case). As output, the following number of tuples were obtained:

$$\begin{tabular}{|c|c|}
  \hline
  Form of $A_2$ & Number of tuples $(A_3,A_4,A_5,A_6)$ \\
  \hline
  (1) & 6811 \\
  (2) &  6811\\
  (3) &  6811\\
  (4) & 7866\\
  (5) &  7866\\
  (6) & 6811 \\
  (7) & 7866 \\
  (8) & 7866\\
  \hline
\end{tabular}$$

These tuples were later processed by a classification algorithm.

\vspace{0.3cm}\hrule\vspace{0.3cm}
\centerline{{\bf Algorithm 4:} Isomorphism classification algorithm}
\vspace{0.3cm}\hrule\vspace{0.3cm}

\textbf{$\bullet $ Input:} A collection of lists of matrices representing semifields

\textbf{$\bullet$ Output:} Representatives of all isomorphism classes present in the collection

\textbf{$\bullet$ Procedure:}

$\hbox{ }$ $S \leftarrow \emptyset$

\textbf{$\hbox{ }$ for} each list $L$ of matrices in the collection \textbf{do}

 \textbf{\hspace{.7cm} if} $L \not\in S$ \textbf{then}

        \textbf{\hspace{1.3cm}} print $L$
        
      \textbf{\hspace{1.3cm}}  Insert in $S$ all cyclic representations of $L$
      
       \textbf{\hspace{.7cm} end}
     
     \textbf{$\hbox{ }$ end}
\vspace{0.3cm}\hrule\vspace{0.3cm}

The correctness of the algorithm is guaranteed by
\cite{Hentzel}[Theorem 2], since any primitive semifield is necessarily cyclic.
The algorithm produced 2826 isomorphism classes that were later processed by another algorithm which classified semifields under isotopy and $S_3$-action. The output of this third algorithm showed the existence of 27 isotopy classes and 12 different planes (up to $S_3$-equivalence). 
\vspace{0.3cm}\hrule\vspace{0.3cm}
\centerline{{\bf Algorithm 5:} Isotopy and $S_3$-action classification algorithm}
\vspace{0.3cm}\hrule\vspace{0.3cm}

\textbf{$\bullet $ Input:} A collection of lists of matrices representing semifields

\textbf{$\bullet$ Output:} Representatives of all isotopy and $S_3$-action classes present in the collection

\textbf{$\bullet$ Procedure:}

$\hbox{ }$ $S \leftarrow \emptyset$

\textbf{$\hbox{ }$ for} each list $L$ of matrices in the collection \textbf{do}

 \textbf{\hspace{.7cm} if} $L \not\in S$ \textbf{then}

        \textbf{\hspace{1.3cm}} print $L$
                
      \textbf{\hspace{1.3cm}}  
	$T \leftarrow \emptyset$
      
       \textbf{\hspace{1.3cm} if} classification is under $S_3$-action \textbf{then}

	      \textbf{\hspace{1.6cm} for} each $\sigma$ in $S_3$ \textbf{do}  

       	\textbf{\hspace{1.9cm}} Insert $\sigma(L)$ in $T$
	
	\textbf{\hspace{1.6cm} end} 

       \textbf{\hspace{1.3cm} else}
       
       \textbf{\hspace{1.6cm}} Insert $L$ in $T$
       
       \textbf{\hspace{1.3cm} end}
       
              \textbf{\hspace{1.3cm} for} each $I$ in $T$ \textbf{do}
              
              \textbf{\hspace{1.6cm} for} each principal isotope $J$ of $I$ \textbf{do}
              
              \hspace{1.9cm} Insert in $S$ all cyclic representations of $J$
              
               \textbf{\hspace{1.6cm} end}
               
                              \textbf{\hspace{1.3cm} end}
                              
                                             \textbf{\hspace{.7cm} end}
     
     \textbf{$\hbox{ }$ end}

\vspace{0.3cm}\hrule\vspace{0.3cm}

At this point we obtained a complete classification of the {known} semifield planes of order 81. We explicitely constructed all finite semifields of such an order that belong to any of the families collected in \cite{Kantor}. Classification under isomorphism, isotopy and $S_3$-action showed the existence of 7 {known} planes (up to $S_3$-action), that we list below. As before, a semifield representative is given for each plane, together with the order of its automorphism group (for a detailed description of the constructions, see \cite{Kantor}).

\vspace{0.3cm}\hrule\vspace{0.3cm}

\centerline{\textbf{I (Desarguesian plane)}}

Finite field $\hbox{GF}(81)$ (4 automorphisms) 
$$(A_2,A_3,A_4)=
(19792, 8866, 186745)$$

\centerline{\textbf{II (Twisted field plane)}}

Twisted Field (1 automorphism) with parameters:
\begin{itemize}
	\item $j\in \hbox{GF}(81)$ such that $j^4+2j+2=0$;
	\item $\alpha\in \hbox{Aut}(\hbox{GF}(81))$ such that $\alpha(x)=x^3$, for all $x\in \hbox{GF}(81)$;
	\item $\beta\in \hbox{Aut}(\hbox{GF}(81))$ such that $\beta(x)=x^{3^3}$, for all $x\in \hbox{GF}(81)$.
\end{itemize}
$$(A_2,A_3,A_4)=
(19792, 30332, 214473)$$

\centerline{\textbf{III (Commutative semifield)}}

Dickson's commutative semifield (4 automorphisms) with parameters:
\begin{itemize}
	\item $k\in \hbox{GF}(9)$ such that $k^2+2k+2=0$;
	\item $\sigma\in \hbox{Aut}(\hbox{GF}(9))$ such that $\sigma(x)=x^3$, for all $x\in \hbox{GF}(9)$.
\end{itemize}
$$(A_2,A_3,A_4)=
(19818, 9001, 355161)$$

\centerline{\textbf{IV}}

Knuth's semifield of type 1 (8 automorphisms) with parameters:
\begin{itemize}
	\item $k\in \hbox{GF}(9)$ such that $k^2+1=0$;
	\item $\alpha\in \hbox{Aut}(\hbox{GF}(9))$ such that $\alpha(x)=x^3$, for all $x\in \hbox{GF}(9)$.
	\item $\beta\in \hbox{Aut}(\hbox{GF}(9))$ such that $\beta(x)=x$, for all $x\in \hbox{GF}(9)$.
	\item $\gamma\in \hbox{Aut}(\hbox{GF}(9))$ such that $\gamma(x)=x^3$, for all $x\in \hbox{GF}(9)$.
\end{itemize}
$$(A_2,A_3,A_4)=
(19794, 428919, 473210)$$

\centerline{\textbf{V}}

Knuth's semifield of type 1 (4 automorphisms) with parameters:
\begin{itemize}
	\item $k\in \hbox{GF}(9)$ such that $k^2+2k+2=0$;
	\item $\alpha\in \hbox{Aut}(\hbox{GF}(9))$ such that $\alpha(x)=x^3$, for all $x\in \hbox{GF}(9)$.
	\item $\beta\in \hbox{Aut}(\hbox{GF}(9))$ such that $\beta(x)=x$, for all $x\in \hbox{GF}(9)$.
	\item $\gamma\in \hbox{Aut}(\hbox{GF}(9))$ such that $\gamma(x)=x^3$, for all $x\in \hbox{GF}(9)$.
\end{itemize}
$$(A_2,A_3,A_4)=
(19801, 191026, 186259)$$

\centerline{\textbf{VI}}

Knuth's semifield of type 2 (2 automorphisms) with parameters:
\begin{itemize}
	\item $f\in \hbox{GF}(9)$ such that $f+2=0$;
	\item $g\in \hbox{GF}(9)$ such that $g+2=0$;
	\item $\sigma\in \hbox{Aut}(\hbox{GF}(9))$ such that $\sigma(x)=x^3$, for all $x\in \hbox{GF}(9)$.
\end{itemize}
$$(A_2,A_3,A_4)=
(19794, 409289, 130416)$$

\centerline{\textbf{VII}}

Knuth's semifield of type 2 (1 automorphism) with parameters:
\begin{itemize}
	\item $f\in \hbox{GF}(9)$ such that $f^2+2f++2=0$;
	\item $g\in \hbox{GF}(9)$ such that $g+2=0$;
	\item $\sigma\in \hbox{Aut}(\hbox{GF}(9))$ such that $\sigma(x)=x^3$, for all $x\in \hbox{GF}(9)$.
\end{itemize}
$$(A_2,A_3,A_4)=
(19794, 519711, 29089)$$

\vspace{0.3cm}\hrule\vspace{0.3cm}

Besides these 7 {known} planes, our results showed the existence of 5 new planes:

\vspace{0.3cm}\hrule\vspace{0.3cm}

\centerline{\textbf{VIII}}

Semifield (4 automorphisms) with tuple of matrices
$$(A_2,A_3,A_4)=
(19825, 253482, 243782)$$

\centerline{\textbf{IX }}

Semifield (1 automorphism) with tuple of matrices 
$$(A_2,A_3,A_4)=
(19792, 8841, 198942)$$

\centerline{\textbf{X }}

Semifield (1 automorphism) with tuple of matrices 
$$(A_2,A_3,A_4)=
(19792, 8956, 202821)$$

\centerline{\textbf{XI }}

Semifield (1 automorphism) with tuple of matrices 
$$(A_2,A_3,A_4)=
(19792, 8956, 408532)$$

\centerline{\textbf{XII}}

Semifield (1 automorphism) with tuple of matrices 
$$(A_2,A_3,A_4)=
(19792, 8984, 461005)$$

\vspace{0.3cm}\hrule\vspace{0.3cm}

Matrices are encoded as numbers according to the following rule: the last three columns of
a matrix $A_i$, having a one in the $i$-th position of the first
column, and zeroes everywhere else

$$\left(%
\begin{array}{ccc}
   a_{11} & a_7 & a_3 \\
   a_{10} & a_6 & a_2 \\
   a_{9} & a_5 & a_1 \\
   a_{8} & a_4 & a_0 \\
\end{array}%
\right)$$
$\hbox{are given as the number} \sum_{i=0}^{1}a_i3^i.$
For each of the 12 semifield representatives, we computed the order of the autotopy group, the list of all principal isotopes, and the order of their isomorphism groups. All these data are collected in Table 1.

\begin{table}[htdp]
\begin{center}
\begin{tabular}{|c|c|c|c|}
  \hline
    \textbf{Plane number}  & $\mathbf{S_3-orbit}$ & \textbf{Order of At} & \textbf{S/A sum}\\
  \hline
  \textbf{I}  & \includegraphics[width=1cm, height=1cm]{hexagonCompGr.pdf} & 25600 & $\frac{1}{4}$\\

  \hline
   \textbf{II} & \includegraphics[width=1cm, height=1cm]{hexagonLinTr.pdf}  &  640 & $\frac{10}{1}$\\
   \hline
   \textbf{III}& \includegraphics[width=1cm, height=1cm]{hexagonLinGr.pdf}&  512 & $\frac{12}{1}+\frac{2}{4}$\\
   \hline
   \textbf{IV} & \includegraphics[width=1cm, height=1cm]{hexagonCompTr.pdf} &  2048 & $\frac{6}{2}+\frac{1}{8}$\\
      \hline
      \textbf{V} & \includegraphics[width=1cm, height=1cm]{hexagonCompTr.pdf} &  1024 & $\frac{6}{1}+\frac{1}{4}$\\
       \hline
       \textbf{VI} & \includegraphics[width=1cm, height=1cm]{hexagonCompTr.pdf} &  128 & $\frac{42}{1}+\frac{16}{2}$\\
         \hline
          \textbf{VII} & \includegraphics[width=1cm, height=1cm]{hexagonCompTr.pdf} &  64 & $\frac{100}{1}$\\
    \hline
   \textbf{VIII} & \includegraphics[width=1cm, height=1cm]{hexagonLinTrB.pdf}  &  256 & $\frac{24}{1}+\frac{1}{2}+\frac{2}{4}$\\
       \hline
   \textbf{IX}& \includegraphics[width=1cm, height=1cm]{hexagon.pdf} &  32 & $\frac{200}{1}$\\
      \hline
   \textbf{X} & \includegraphics[width=1cm, height=1cm]{hexagonLinTrB.pdf}   &  32 & $\frac{200}{1}$\\
      \hline
   \textbf{XI} & \includegraphics[width=1cm, height=1cm]{hexagonCompTr.pdf} &  16 & $\frac{400}{1}$\\
     \hline
   \textbf{XII} & \includegraphics[width=1cm, height=1cm]{hexagonLinTrB.pdf}   &  64 & $\frac{100}{1}$\\
          \hline
\end{tabular}
\caption{Classification of finite semifields of 81 elements (last five are new)}
\end{center}
\end{table}%

These results show the existence of 27 nonisomorphic planes. Two of them are commutative: the Desarguesian one (I) and the one coordinatized by Dickson's commutative semifield (III). The total amount of commutative semifields of order 81 is three: the finite field $GF(81)$, Dickson's commutative semifield, and one of its isotopes with matrices $(A_1,A_2,A_3,A_4)=(59293, 19818, 12291, 359225)$.
Table 2 contains a summary of our results.

\begin{table}[htdp]
\begin{center}
\begin{tabular}{|c|c|c|c|}
  \hline
   \textbf{Number of (commutative) classes}  & \textbf{Isomorphism} & \textbf{Isotopy} & \textbf{$S_3$-action}\\
  \hline
  Previously known &  245 (3) & 11 (2) & 7 (2)\\

  \hline
   Actual number & 2826 (3) & 27 (2) &12 (2)\\
   \hline
\end{tabular}
\caption{81-element finite semifields}
\end{center}
\end{table}%

\section{Concluding remarks}

In this paper we study finite semifields of order 81, with the help of computational tools.  We obtain a complete classification of finite semifields of such an order 81 and their planes (Tables 1 and 2). The total amount of these planes is 27. Two of them are commutative (coordinatized by three nonisomorphic commutative finite semifields), approximately half of them (16) were previously unknown.

\end{document}